\documentclass{amsart}
\usepackage{hyperref}
\usepackage{mathabx}
\usepackage{amsmath, amsthm, amssymb}
\usepackage{mdwlist} 
\usepackage{graphicx}
\usepackage{ytableau}
\usepackage[OT2,T1]{fontenc}
\usepackage[utf8]{inputenc}
\usepackage[russian,english]{babel}

\usepackage[author-year]{amsrefs}
\usepackage{enumerate}
\usepackage[nomain,symbols]{glossaries}
\usepackage[linesnumbered,boxruled]{algorithm2e}

\usepackage[mathscr]{euscript}
\usepackage{mathrsfs}

\usepackage{pictexwd,dcpic}

\newtheorem{theorem}{Theorem}[section]

\newtheorem{corollary}[theorem]{Corollary}

\newtheorem{lemma}[theorem]{Lemma}

\theoremstyle{definition}
\newtheorem{definition}[theorem]{Definition}

\theoremstyle{remark}
\newtheorem{remark}[theorem]{Remark}
\newtheorem{example}[theorem]{Example}

\DeclareMathOperator{\sign}{sgn}

\definecolor{absolutezero}{rgb}{0.0,0.28,0.73}
\definecolor{darkgreen}{rgb}{0.0,0.38,0.0}

\newcommand{\changed}[1]{{#1}}
\newcommand{\refonly}[1]{}

\DeclareMathOperator*{\bigE}{\mathbb E}

\newcommand{\expected}[2]{
\bigE_{\substack{#1}}\left({#2}\right)}

\newcommand{\conv}[1]{\mathrm{Conv}(#1)}

\newcommand{\vol}{\mathrm{Vol}}
\newcommand{\dd}{\,\mathrm{d}}

\newcommand{\binomial}[2]{\ensuremath{\left( \begin{matrix}#1 \\ #2 \end{matrix} \right)}}

\newcommand{\defun}[5]{\ensuremath{\begin{array}{lrcl}
#1:&#2 & \longrightarrow & #3\\&#4 & \longmapsto & #5\end{array}}}

\newcommand{\defeq}{\stackrel{\mathrm{def}}{=}}
\newcommand{\diag}[1]{\mathrm{diag}\left(#1\right)}

\author{Gregorio Malajovich}
\title[On the expected number of real roots...]{On the expected number of real roots of polynomials and exponential sums}
\date{Apr 12, 2022}
\address{Departamento de Matemática Aplicada, Instituto de Matemática, Universidade Federal do
Rio de Janeiro. Caixa Postal 68530, Rio de Janeiro, RJ 21941-909, Brasil.}
\email{gregorio@im.ufrj.br}
\thanks{This research was partially funded by the {\em Coordenação de Aperfeiçoamento de Pessoal de Nível Superior} (CAPES), grants PROEX and PRINT, and by the {\em Fundação Carlos Chagas Filho de Amparo à Pesquisa do Estado do Rio de Janeiro} (FAPERJ)}

\begin{document}
\refonly{
\thispagestyle{empty}
Sydney NSW, Nov 23rd, 2022.
\bigskip

Dear Referees,
\medskip

First of all, thanks for the carful reading of this paper and for the thoughtful comments. Most suggestions were incorporated in the text, and are marked in \changed{red}. The two major suggestions of Referee X are nor Remarks 1.7 and 2.3. 
\medskip

The only exception is comment 25 from referee Y. Chicago's author-year style mandates to spell the first author of each reference by last name, first name. The other authors are spelled first name, last name. This style was incorporated by the American Mathematical Society in the {\tt amsrefs} package. Therefore I could not change the actual ordering of first and last name.
\medskip

In case the publisher as a different policy for the references, this can be addressed at the production stage.
\bigskip

Sincerely yours,
\medskip

Gregorio Malajovich

\setcounter{page}{0}
\newpage
}
\maketitle
\begin{abstract}
The expected number of real projective roots of orthogonally invariant random homogeneous real polynomial systems is known to be equal to the square root of the Bézout number. A similar result is known for random multi-homogeneous systems, invariant through a product of orthogonal groups.
In this note, those results are generalized to certain families of
sparse polynomial systems, with no orthogonal invariance assumed. 
\end{abstract}

\section{Introduction}
\subsection{Motivation and outline}
The Kostlan-Shub-Smale ensemble is the probability space of $n+1$-variate systems of 
polynomial equations of the form
\begin{eqnarray*}
	f_1(\mathbf X) &=& \sum_{|\mathbf a| = d_1} f_{1,\mathbf a}\, \sqrt{\binomial{d_1}{\mathbf a}} \mathbf X^{\mathbf a} 
\\
	&\vdots&\\
	f_n(\mathbf X) &=& \sum_{|\mathbf a| = d_n} f_{n,\mathbf a}\, \sqrt{\binomial{d_n}{\mathbf a}} \mathbf X^{\mathbf a} 
\end{eqnarray*}
with $f_{i,\mathbf a} \sim \mathscr N(0,1)$ i.i.d. real standard normals. 
Above, $\mathbf X^{\mathbf a} = X_0^{a_0} X_1^{a_1} \dots X_n^{a_n}$ and 
$\binomial{d}{\mathbf a} = \frac{d!}{a_0! \dots a_n!}$ is the coefficient of $\mathbf X^{\mathbf a}$ in $(\sum X_i)^d$.

\ocite{Kostlan} and \ocite{Bezout2} proved that the expected number of
projective real roots of a random polynomial from the ensemble above is
$\sqrt{d_1 d_2 \dots d_n}$, that is the square root of the 
generic number of complex projective roots provided by Bézout's Theorem.
The choice of the metric coefficients
$\sqrt{\binomial{d}{\mathbf a}}$ guarantees the invariance of the
probability distribution with respect to orthogonal transformations
\changed{\cite{BCSS}*{Theorem~1 p.218}}.
This is an important step to establish the expected number of roots.
\changed{It is also known that} if we replace $X_0$ by $1$ and $a_0$ by $d-a_1-\dots -a_n$, the 
expected number of real roots remains the same.

\ocite{Rojas-Park-City} generalized this result to systems of
multi-homogeneous polynomial equations, with a proper choice of
metric coefficients.
This choice provides a probability distribution that is 
invariant through the action of a product of orthogonal groups,
eventually after multi-homogeneization.
\medskip

{\em Sparse} polynomial systems are 
systems of polynomials that are not assumed to be dense or multi-homogeneous.
\ocite{Bernstein} found out that the generic number of roots
over \changed{$(\mathbb C \setminus \{0\})^n$} is $n!$ times the mixed volume of the convex
hulls of the supports. In many cases, Bernstein's bound is much smaller than the Bézout number. This rises
the question, is it possible to bound the expected number of real roots
for sparse systems in terms of the mixed volume, or of another geometric
invariant of the supports?
Very little is known, 
except for a local inequality given by \ocite{MRHigh}*{Theorem 3}. 

Below, the Kostlan-Shub-Smale
bound is generalized to certain families of sparse polynomial systems
where the supports have the same {\em shape}. In that case,
the expected number of real solutions grows like the square root of the 
product of the scaling factors. It can also be compared to the
square root of the generic number of complex roots.
The metric coefficients that make this possible
come from \ocite{Aronszajn}'s multiplication of inner product spaces. 

\subsection{Sparse polynomials and exponential sums}
Let $A \changed{\, \subset \,} \mathbb Z^n$ be a finite set (the {\em support}), and $\alpha: A \rightarrow (0, \infty)$ be arbitrary metric coefficients.
\changed{For $\mathbf a \in A$, write $\mathbf X^{\mathbf a} = X_1^{a_1} X_2^{a_2} \dots X_n^{a_n}$ and $e^{\mathbf a \mathbf x} = e^{a_1 x_1 + \dots + a_n x_n}$.}

To each pair $(A, \alpha)$ we can associate the real inner product space $\mathscr P_{A,\alpha}$ of polynomials, respectively the real inner product space $\mathscr E_{A,\alpha}$ of exponential sums, with orthonormal basis
\changed{$\{\dots,\alpha_{\mathbf a} \mathbf X^{\mathbf a}, \dots\}_{\mathbf a \in A}$},
respectively \changed{$\{\dots,\alpha_{\mathbf a} e^{\mathbf a \mathbf x}, \dots\}_{\mathbf a \in A}$}. Those spaces \changed{correspond to functions with closely related domains}: if $\mathscr D$ is a 
measurable subset of $\mathbb R^n$ and $\exp(\mathscr D) \changed{\, \subset \,} \mathbb R_{>0}^n$ its 
componentwise exponentiation, 
the solution set of the system
\begin{equation}\label{expsum}
	\sum_{\mathbf a \in A} f_{i\changed{,}\mathbf a}\, \alpha_{\mathbf a} e^{\mathbf a \mathbf x} = 0 
	\hspace{3em} i=1, \dots, n
\end{equation}
for $\mathbf x \in \mathscr D$ is mapped by exponentiation $\mathbf x \mapsto \mathbf X = \begin{pmatrix} e^{x_1} \\ \vdots \\ e^{x_n} \end{pmatrix}$ onto the
solution set of the system
\begin{equation}\label{sparse-poly}
	\sum_{\mathbf a \in A} f_{ \mathbf a}\, \alpha_{\mathbf a} \mathbf X^{\mathbf a} = 0
	\hspace{3em} i=1, \dots, n,
\end{equation}
for $X \in \exp(\mathscr D)$.
In the particular case $\mathscr D=\mathbb R^n$, $\exp(\mathscr D)$ is the
(strictly) positive orthant. In order to state our main results, it is
convenient to restrict Aronszajn's multiplication of reproducing kernel spaces
to the particular case of spaces of polynomials (resp. exponential sums) with monomial basis.

\begin{definition}[Aronszajn multiplication, restricted case]
\[
\mathscr P_{A,\alpha}= 
\mathscr P_{B, \beta} \mathscr P_{C,\gamma}
	\hspace{1em} \text{and}\hspace{1em}  
\mathscr E_{A,\alpha}= 
\mathscr E_{B, \beta} \mathscr E_{C,\gamma} 
\]
if and only if $A = B+C$ and for each $\mathbf a \in A$,
\[
\gamma_{\mathbf a}^2 = \sum_{\mathbf b + \mathbf c = \mathbf a}
\alpha_{\mathbf b}^2 \beta_{\mathbf c}^2.
\]
\end{definition}
Inductively, the $d$-th power of a space is defined by
$\mathscr P_{A,\alpha}^d = 
\mathscr P_{A,\alpha} \mathscr P_{A,\alpha}^{d-1}$. Its support is not
$d A$ in general, 
but the convex hull of its support is $d\, \conv{A}$. In this sense,
we say that $A$ and the support of $\mathscr P_{A, \alpha}^d$ have {\em the same shape}.

\begin{example}\label{Kostlan-ensemble}
Let $\mathscr P$ be the space with orthonormal basis $1, X_1, \dots, X_n$.
Then $\mathscr P^d$ has orthonormal basis
	\[\changed{
	\left\{ \dots, 
	\sqrt{\binomial{d}{\mathbf a}} \mathbf X^{\mathbf a},
	\dots \right\}_{|\mathbf a| \le d}
	}\]
and is isomorphic to the space of homogeneous $n+1$-variate 
polynomials of degree
$d$ with the Weyl-Bombieri inner product.
\end{example}

\subsection{Main results}
Suppose that $\mathscr F_1, \dots, \mathscr F_n$ are spaces
of exponential sums, $\mathscr F_i = \mathscr E_{A_i, \alpha_i}$. 
We denote by
$E_{\mathscr D}(\mathscr F_1, \dots, \mathscr F_n)$ the expected number
of roots in $\mathscr D$ of a system of exponential sums of the form
\begin{equation} \label{mixed}
\sum_{\mathbf a \in A_i} 
\mathbf f_{i \mathbf a} \alpha_{i \mathbf a} e^{\mathbf a \mathbf x}
 = 0 \hspace{3em} i=1,\dots, n,
\end{equation}
for $\mathbf f_{i\changed{,}\mathbf a} \sim \mathscr N(0,1)$ i.i.d. standard normal random
real numbers. In case $\mathscr F_1 = \dots = \mathscr F_n = \mathscr F$,
we write just $E_{\mathscr D}(\mathscr F)$.
\medskip
\par
\begin{theorem}\label{gen-square-root}
	Let $\mathscr D$ be a measurable set of $\mathbb R^n$.
	Assume that $\mathscr F_1$, \dots, $\mathscr F_n$ are spaces
	of exponential sums, and that $d_1, \dots, d_n \in \mathbb N$.
	Then,
\[
	E_{\mathscr D}( \mathscr F_1^{d_1}, \dots, \mathscr F_n^{d_n}) 
= \sqrt{d_1 d_2 \dots d_n} 
	E_{\mathscr D}(\mathscr F_1, \dots, \mathscr F_n).
\]
\end{theorem}

\begin{corollary}
	Let $\mathscr D$ be a measurable set of $\mathbb R^n$.
	Assume that $\mathscr P_1$, \dots, $\mathscr P_n$ are spaces
	of polynomials, and that $d_1, \dots, d_n \in \mathbb N$.
	Then,
\[
	E_{\exp(\mathscr D)}( \mathscr P_1^{d_1}, \dots, \mathscr P_n^{d_n}) 
= \sqrt{d_1 d_2 \dots d_n} 
	E_{\exp(\mathscr D)}(\mathscr P_1, \dots, \mathscr P_n).
\]
\end{corollary}

If $\mathbf x$ is a complex root of \eqref{mixed}, then for all
$\mathbf k \in \mathbb Z^n$, $\mathbf x + 2 \pi \sqrt{-1} \mathbf k$
is also a complex root of \eqref{mixed}. Therefore it makes sense to count the
number of complex roots in the quotient space
$\mathbb C^n \mod 2 \pi \sqrt{-1} \mathbb Z^n$.
We denote by $E_{\mathscr D \times (-\pi,\pi]^n}(\mathscr F_1, \dots, \mathscr F_n)$ this generic number of roots.
As explained by \ocite{Fewspaces}, Aronszajn's multiplication gives the set 
of reproducing \changed{spaces} a semigroup structure. \changed{For each $i$,
assume $\mathcal F_j$ fixed for all $j \ne i$. The  
mapping
\[
\mathcal F_i \mapsto
E_{\mathscr D \times (-\pi,\pi]^n}(\mathscr F_1, \dots, \mathscr F_n) 
\]
giving the generic number of roots is
a homomorphism from the semi-group of reproducing kernel spaces
into the additive semi-group of positive real numbers.} Hence,
\[
E_{\mathscr D \times (-\pi,\pi]^n}(\mathscr F_1^{d_1}, \dots, \mathscr F_n^{d_n})
=
d_1 d_2 \dots d_n
E_{\mathscr D \times (-\pi,\pi]^n}(\mathscr F_1, \dots, \mathscr F_n)
\]
\begin{corollary}
	Let $\mathscr D$ be a measurable set of $\mathbb R^n$.
	Assume that $\mathscr F_1$, \dots, $\mathscr F_n$ are spaces
	of exponential sums, with $E_{\mathscr D \times (-\pi,\pi]^n}(\mathscr F_1^{d_1}, \dots, \mathscr F_n^{d_n}) \ne 0$.
	Then,
\[
	\frac{E_{\mathscr D}( \mathscr F_1^{d_1}, \dots, \mathscr F_n^{d_n})}
	{\sqrt{E_{\mathscr D \times (-\pi,\pi]^n}(\mathscr F_1^{d_1}, \dots, \mathscr F_n^{d_n})}}
\]
is independent of $d_1, \dots, d_n$.
\end{corollary}

Also, we show in this note that the expected number of real roots is sub-additive \changed{with respect to Aronszajn multiplication}:

\begin{theorem}\label{sub-additive}
	Let $\mathscr D$ be a measurable set of $\mathbb R^n$.
	Assume that $\mathscr F_1$, \dots, $\mathscr F_{n-1}$,
	$\mathscr G$ and $\mathscr H$ are spaces
	of exponential sums, with $E_{\mathscr D \times (-\pi,\pi]^n}(\mathscr F_1^{d_1}, \dots, \mathscr F_n^{d_n}) \ne 0$.
	\[
		E_{\mathscr D}(\mathscr F_1, \dots, \changed{\mathscr F_{n-1}}, \mathscr G \mathscr H) \le
	E_{\mathscr D}(\mathscr F_1, \dots, \mathscr F_{n-1}, \mathscr G)+
	E_{\mathscr D}(\mathscr F_1, \dots, \mathscr F_{n-1}, \mathscr H).
\]
\end{theorem}

\changed{
\begin{remark}
It is easy to produce spaces as products of inner product spaces. However, decomposing a given space as a product may be an extremely difficult problem. The results in this paper assume that such a product decomposition is available.
\end{remark}}
\subsection{Sign conditions}
There is nothing particular about positive zeros of polynomials. 
To each possible {\em sign condition} $\mathbf s \in \{ \pm 1\}^n$, we associate the orthant $\mathbb R^n_{\mathbf s} = \{
\mathbf X \in \mathbb R^n: \sign(\mathbf X) = \mathbf s \}$.
Let $W=\changed{\, \bigcup\, } W_{\mathbf s}$ be measurable, where $W_{\mathbf s} \changed{\, \subset \,} \mathbb R^n_{\mathbf s}$.
Also, let 
$|W_{\mathbf s}| = 
\diag{\mathbf s} W_{\mathbf s} \changed{\, \subset \,} \mathbb R_{>0}^n$.
Assume $W$ measurable. That is, all the
$W_{\mathbf s}$ are measurable.
\begin{lemma} Let $W \subset \mathbb R_{\ne 0}^n$.
	Let $\mathscr P_1 = \mathscr P_{A_1, \alpha_1}$, \dots, $\mathscr P_n = \mathscr P_{A_n, \alpha_n}$ be spaces of polynomials. Then,
\[
	E_{W}(\mathscr P_1, \dots, \mathscr P_n) =
	\sum_{\changed{\mathbf s \in \{ \pm 1 \}^n}} 
	E_{|W_{\mathbf s}|}(\mathscr P_1, \dots, \mathscr P_n)
.
\]
\end{lemma}
\begin{proof}
	Because $W$ is the disjoint union of the $W_{\mathbf s}$,
\[
E_{W}(\mathscr P_1, \dots, \mathscr P_n) =
	\sum_{\changed{\mathbf s \in \{ \pm 1 \}^n}} 
	E_{W_{\mathbf s}}(\mathscr P_1, \dots, \mathscr P_n)
.
\]
For each sign condition $\mathbf s$,
we can map a solution pair $(\mathbf f, \diag{\mathbf s} \mathbf X)$, $X \in \mathbb R_{>0}^n$ for
Equation \eqref{sparse-poly}
into another solution pair $(\mathbf g, \mathbf X)$
with
\[
g_{i,\mathbf a} = \mathbf s^{\mathbf a} f_{i,\mathbf a}
.
\]
	Since the $f_{i \mathbf a}$ have symmetric, independent distributions, \changed{we have}
\[
	E_{W_{\mathbf s}}(\mathscr P_1, \dots, \mathscr P_n)
=
	E_{|W_{\mathbf s}|}(\mathscr P_1, \dots, \mathscr P_n)
.
\]
\end{proof}

\begin{corollary}
	Let $\mathscr W$ be a measurable set of $\mathbb R_{\ne 0}^n$.
	Assume that $\mathscr P_1$, \dots, $\mathscr P_n$ are spaces
	of polynomials, and that $d_1, \dots, d_n \in \mathbb N$.
	Then,
\[
	E_{\mathscr W}( \mathscr P_1^{d_1}, \dots, \mathscr P_n^{d_n}) 
= \sqrt{d_1 d_2 \dots d_n} 
	E_{\mathscr W}(\mathscr P_1, \dots, \mathscr P_n).
\]
\end{corollary}

\section{Vitale's theorem and generalization}

The proof of the main results will use a generalization of a theorem
stated by 
\ocite{Vitale}. 

\begin{theorem}\label{Vitale}
Let $\mathscr R$ be a probability measure in $\mathbb R^n$ with
$\expected{\mathbf r \simeq \mathscr R}{\|\mathbf r\|}< \infty$.
Then,
	\[
		\expected{\mathbf r_1, \dots, \mathbf r_n \sim \mathscr R \text{ i.i.d.}}{\left| \det \begin{pmatrix} &\mathbf r_1& \\ &\vdots& \\ &\mathbf r_n& \end{pmatrix} \right|}
	= n! \vol_n{C}
	\]
	where $C$ is the convex body with support function
\[
\defun{h_C}{S^{n-1}}{\mathbb R}{\mathbf u}{h_C(\mathbf u) = 
	\expected{\mathbf r \sim \mathscr R}{h_{[0,\mathbf r]}(\mathbf u)}.}
\]
\end{theorem}

Before the generalization, we review a few 
\changed{related identities.} If $\mathscr R \sim -\mathscr R$, 
\[
	h_C(\mathbf u) = \frac{1}{2}\expected{\mathbf r \sim \mathscr R}{\left| \langle \mathbf u, \mathbf r \rangle \right|}.
\]
In case $\mathbf r \sim \mathscr N(0, \Sigma^2)$, $\langle \mathbf u, \mathbf r \rangle \sim \mathscr N(0, \|\Sigma \mathbf u\|^2)$ and $|\langle \mathbf u, \mathbf r \rangle |$ is a {\em half-normal} of average $\sqrt{\frac{2}{\pi}} \| \Sigma \mathbf u\|$. Hence, 
\begin{equation}\label{volC}
	\vol_n(C) = (2\pi)^{-\frac{n}{2}} \det(\Sigma) \vol_n B^n.
\end{equation}

\ocite{BBLM}*{Theorem 5.4} generalized Vitale's theorem to
the situation where blocks of rows are i.i.d, and each block
is independent. A particular case which they trace back to 
Wolfgang Weil \ycite{Weil76}*{Theorem 4.2} is of particular
interest here:

\begin{theorem}\label{WWeil}
Let $\mathscr R_1, \dots, \mathscr R_n$ be probability measures in $\mathbb R^n$ with
	$\expected{\mathbf r \simeq \mathscr R_i}{\|\mathbf r\|}< \infty$ \changed{for all $i$.}
Then,
	\[
		\expected{\mathbf r_i \sim \mathscr R_i \text{ indep.}}{\left| \det \begin{pmatrix} &\mathbf r_1& \\ &\vdots& \\ &\mathbf r_n& \end{pmatrix} \right|}
			= n! MV(C_1, \dots, C_n)
	\]
	where $C_i$ is the convex body with support function
\[
	\defun{h_{C_i}}{S^{n-1}}{\mathbb R}{\mathbf u}{h_{C_i}(\mathbf u) = 
	\expected{\mathbf r \sim \mathscr R_i}{h_{[0,\mathbf r]}(\mathbf u)}.}
\]
\end{theorem}

\changed{
\begin{remark}
An anonymous referee pointed out that 
Theorem 5.4 in \ocite{BBLM} assumes that the rows $r_i$ are divided 
into independent blocks. Therefore,
this theorem can possibly provide a root
count for random systems partitionned
into independent blocks of equations. As this generality 
exceeds the objectives of this paper, this avenue of research 
will not be pursued here.
\end{remark}
}

\section{The real toric variety}

\subsection{Notations.}
We will use the notation $\mathbf f_i$ for the row vector $[\dots, f_{i\changed{,}\mathbf a}, \dots]_{\mathbf a \in A}$. The condition that the $f_{i\changed{,}\mathbf a} \sim \mathscr N(0,1; \mathbb R)$ are i.i.d. is equivalent to the requirement that the
$\mathbf f_i \sim
\mathscr N(0,I; \mathbb R^{A_i})$ are independent. Also, we define the column vector
\[
	V_{A,\alpha}(\mathbf x) = 
	\begin{pmatrix} \vdots  \\
		\alpha_{\mathbf a} e^{\mathbf a \mathbf x}
		\\
		\vdots
	\end{pmatrix}_{\mathbf a \in A}
	.
\]
System \eqref{mixed} can now be
written as
\[
\mathbf f_i \mathbf V_{A_i,\alpha_i}(\mathbf x) = 0 \hspace{1em} i=1,\dots, n.
\]
\par

The space $\mathscr E_{A, \alpha}$ is a reproducing kernel space \changed{\cite{Aronszajn}} with 
\[
K_{A, \alpha}( \mathbf x, \mathbf y) = 
\mathbf V_{A, \alpha}( \mathbf y)^* 
\mathbf V_{A, \alpha}( \mathbf x) 
= 
\sum_{\mathbf a \in A}
\alpha_{\mathbf a}^2 e^{\mathbf a (\mathbf x + \mathbf y)}
\]
so that 
\[
	\mathbf f \, \mathbf V_{A,\alpha}(\mathbf x) = \langle \mathbf f(\cdot), 
K_{A,\alpha}(\cdot, \mathbf x) \rangle
.
\]

\subsection{Integral geometry}
Let $\mathscr F = \mathscr E_{A, \alpha}$ with $\mathbf V=\mathbf V_{A,\alpha}$. The basic construction in integral
geometry gives the expected number of zeros
$E_{\mathscr D}(\mathscr F)$ in terms of the $n$-dimensional volume of 
$\mathscr V = \{[ \mathbf V(\mathbf x)]: \mathbf x \in \mathcal D \changed{\, \subset \,} \mathbb R^n \} \changed{\, \subset \,} \mathbb P(\mathbb R^A)$. More precisely

\begin{theorem}\label{integral1}
	\[
		E_{\mathscr D}(\mathscr F) = 
		\frac{\vol_n (\mathscr V)}{\vol_n (\mathbb R\mathbb P^{n})}
	\]
\end{theorem}
Since Theorem~\ref{integral1} \changed{will} follow from a more general statement, its proof is postponed.

\begin{example} Let $\mathscr D=\mathbb R^n$.
Let $\mathscr F$ be the space with 
orthonormal basis $(1, e^{x_1}, \dots, e^{x_n})$.
We produce an isomorphism with 
the space $\mathscr P$
from example~\ref{Kostlan-ensemble} 
by setting $X_i = e^{x_i}$. For convenience, let $X_0=1$.
Then,
	$\mathscr V = \{ (X_0:X_1:\cdots:X_n) \in \mathbb R \mathbb P^n: \changed{(X_i > 0\, \forall i)} \text{ or } \changed{(X_i < 0\, \forall i)} \}$ so $\mathscr E_{\mathbb R_{>0}^n}(\mathscr P)=2^{-n}$. 
\end{example}

\subsection{The mixed real toric variety}

Now suppose that $\mathscr F_i = \mathscr E_{A_i, \alpha_i}$ with
$\mathbf V_i=\mathbf V_{A_i, \alpha_i}$, for $i=1, \dots, n$. In this
case we set
\[
\defun{\mathbf V}{\mathbb R^n}{\mathscr F_1^* \times \dots \times \mathscr F_n^*}{\mathbf x}{\mathbf V(\mathbf x) = (\mathbf V_1(\mathbf x), \dots, \mathbf V_n(\mathbf x))}
\]
and compose it with canonical projections to obtain
\[
	\defun{[\mathbf V]}{\mathbb R^n}{\mathbb P(\mathscr F_1^*) \times \dots \times \mathbb P(\mathscr F_n^*)}{\mathbf x}{[\mathbf V(\mathbf x)] = ([\mathbf V_1(\mathbf x)], \dots, [\mathbf V_n(\mathbf x)])}
.
\]
The mixed Veronese variety $\mathscr V$ is just the topological closure of
$\{ [\mathbf V(\mathbf x)]: \mathbf x \in \mathbb R^n\}$. 
For each $i$, define $\langle \cdot , \cdot \rangle_{i, \mathbf x}$ as the pull-back of the
real Fubini-Study metric of $\mathbb P(\mathscr F_i^*)$ by $[\mathbf V_i](\mathbf x)$. Define also $\|\cdot\|_{i, \mathbf x}$ as the associated metric and
$K_i(\mathbf x, \mathbf y) = \mathbf V_i(\mathbf y)^* \mathbf V_i(\mathbf x)$ the reproducing \changed{kernel}.
\begin{lemma}\label{MV-integral} 
Let $\mathscr F_1, \dots, \mathscr F_n$ be spaces of exponential sums,
and let $\mathscr D$ be measurable in $\mathbb R^n$.
	Then, 
\[
	E_{\mathscr D}
	(\mathscr F_1, \dots, \mathscr F_n) = 
	\frac{n!}{\sqrt{2\pi}^{n}}
	\int_{\mathscr D} 
	\changed{\dd}\mathbb R^n(\mathbf x) 
\mathrm{MV}_n( C_1(\mathbf x), \dots, C_n(\mathbf x))
\]
	where $C_i(\mathbf x)$ is the convex body with support function
\[
	h_{C_i(\mathbf x)}(\mathbf u) = 
\frac{1}{\sqrt{2\pi}} \|\mathbf u\|_{i,\mathbf x}
.
\]
\end{lemma}

\begin{proof}
For short, let $\mathscr F = \mathscr F_1 \times \dots \times \mathscr F_n$.
Let $\mathscr S=\{ (\mathbf f, \mathbf x) \in \mathscr F \times \mathscr D: \mathbf f_i \mathbf V_i(\mathbf x)=0, \ i=1, \dots, n\}$ be the solution variety,
	$\pi_1: \mathscr S \rightarrow \mathscr F$ and $\pi_2: \mathscr S \rightarrow \mathbb R^n$ the canonical projections. Let $\mathscr F_{\mathbf x} = \pi_1 \circ \pi_2^{-1}(\mathbf x)$ be the fiber above $\mathbf x$.  The coarea formula 
	\cites{BCSS,Malajovich-nonlinear} yields
\begin{eqnarray*}
	E(\mathscr F_1, \dots, \mathscr F_n) &=& 
	\int_{\mathscr F} \# \pi_1^{-1} (\mathbf f) \,
	\frac{1}{\sqrt{2\pi}^{\dim(\mathscr F)}}
	\dd\mathscr F(\mathbf f) \\
	&=&
	\int_{\mathscr D} 
	\changed{\dd}\mathbb R^n(\mathbf x)\, 
	\int_{\mathscr F_{\mathbf x}}
	NJ^{-1}(\mathbf f, \mathbf x)\,
	\frac{1}{\sqrt{2\pi}^{\dim(\mathscr F)}}
	\dd\mathscr F_{\mathbf x}(\mathbf f) 
\end{eqnarray*}
	where $NJ$ is the normal Jacobian of the local implicit function $G:T_{\mathbf f} \mathscr F \rightarrow T_{\mathbf x} \mathbb R^n$,
	$\mathbf x= G(\mathbf f)$.
By differentiating
	$\mathbf f_{it} \mathbf V_i(\mathbf x_t) \equiv 0$
	at $t=0$,
\[
	\mathbf f_i D\mathbf V_i(\mathbf x)
	\dot {\mathbf x} =
	- \dot {\mathbf f_i} \mathbf V_i(\mathbf x), \hspace{1em} i=1, \dots, n.
\]
This is the same as:
\[
	\frac{1}{\| \mathbf V_i(\mathbf x)\|} 
	\mathbf f_i D\mathbf V_i(\mathbf x)
	\dot {\mathbf x} =
	- 
	\frac{1}{\|\mathbf V_i(\mathbf x)\|} K_i(\cdot, \mathbf x)^*
	(\dot {\mathbf f_i}), \hspace{3em} i=1, \dots, n.
\]
	Denote the condition matrix by 
\[
	M(\mathbf f, \mathbf x) = 
	\begin{pmatrix}
		&\vdots& \\
		&	\frac{1}{\| \mathbf V_i(\mathbf x)\|} 
		\mathbf f_i D\mathbf V(\mathbf x)&\\
		&\vdots&
	\end{pmatrix}
	.
\]
The derivative of the implicit function $G$ can be written as
	\[
		DG(\mathbf f, \mathbf x) =
	-M(\mathbf f, \mathbf x)^{-1} 
	\bigoplus_i \frac{1}{\| V_i(\mathbf x)\|} K_i(\cdot, \mathbf x)^*
	.
\]
The normal Jacobian is therefore
	\changed{
\[
	NJ(\mathbf f, \mathbf x) =
	\sqrt{\det\left(	M(\mathbf f, \mathbf x)^{-1} 
	\diag{\frac{K_i(\mathbf x, \mathbf x)}{\| \mathbf V_i(\mathbf x) \|^2}}
	M(\mathbf f, \mathbf x)^{-*} \right)}
.
\]
Since $\frac{K_i(\mathbf x, \mathbf x)}{\| \mathbf V_i(\mathbf x) \|^2}=1$,
we obtain that
\[
	NJ(\mathbf f, \mathbf x) =
	\left| \det \left( M(\mathbf f, \mathbf x)^{-1} \right) \right|.
\]}
	In order to integrate $NJ(\mathbf f, \mathbf x)^{-1}$ over the fiber $\pi_1 \circ \pi_2(\mathbf x)^{-1}$, we first remark that for $\mathbf f_i$ in the fiber,
\[
	M(\mathbf f, \mathbf x) = 
	\begin{pmatrix}
		&\vdots& \\
		&\mathbf f_i D[\mathbf V_i](\mathbf x)&\\
		&\vdots&
	\end{pmatrix}
	.
\]
For $\mathbf f_i$ orthogonal to the fiber, that is colinear to $\mathbf V_i(\mathbf x)$,
\[
\mathbf f_i D[\mathbf V_i](\mathbf x)
=
\mathbf f_i P_{\mathbf V_i(\mathbf x)^{\perp}} \frac{1}{\|\mathbf V_i(\mathbf x)\|}
D\mathbf V_i(\mathbf x) = 0.
\]
Introducing $n$ extra Gaussian variables,
\begin{eqnarray*}
	E_{\mathscr D}(\mathscr F_1, \dots, \mathscr F_n) &=& 
	\int_{\mathscr D} 
	\changed{\dd}\mathbb R^n(\mathbf x) 
	\int_{\mathscr F_{\mathbf x}}
	NJ^{-1}(\mathbf f, \mathbf x)
	\frac{1}{\sqrt{2\pi}^{\dim(\mathscr F)}}
	\dd\mathscr F_{\mathbf x}(\mathbf f) 
\\
	&=&
	\frac{1}{\sqrt{2\pi}^{n}}
	\int_{\mathscr D} 
	\changed{\dd}\mathbb R^n(\mathbf x) 
	\int_{\mathscr F}
	NJ^{-1}(\mathbf f, \mathbf x)
	\frac{1}{\sqrt{2\pi}^{\dim(\mathscr F)}}
	\dd\mathscr F(\mathbf f) 
\\
	&=&
	\frac{1}{\sqrt{2\pi}^{n}}
	\int_{\mathscr D} 
	\changed{\dd}\mathbb R^n(\mathbf x) 
	\expected{ \mathbf f \sim \mathscr N(0,I; \mathscr F)}{
\left|		\det \begin{pmatrix}
	\vdots \\
		\mathbf f_i D[\mathbf V_i](\mathbf x)\\
	\vdots
\end{pmatrix} \right|}
	.
\end{eqnarray*}
The rows $\mathbf f_i D[\mathbf V_i](\mathbf x)$ are independent, so
we apply Theorem~\ref{WWeil}:
\[
E(\mathscr F_1, \dots, \mathscr F_n) = 
	\frac{n!}{\sqrt{2\pi}^{n}}
	\int_{\mathscr D} 
	\changed{\dd}\mathbb R^n(\mathbf x) \,
\mathrm{MV}_n( C_1(\mathbf x), \dots, C_n(\mathbf x))
\]
where $C_i(\mathbf x)$ is the convex body with support function
\[
	h_{C_i(\mathbf x)}(\mathbf u) = 
	\frac{1}{2} 
	\expected{\mathbf f_i \sim \mathscr N(0,I; \mathscr F_i)}{\left| \langle \mathbf u,
	\changed{(\mathbf f_i D[\mathbf V_i](\mathbf x))^T} \rangle \right|}
=
	\frac{1}{2} 
	\expected{\mathbf f_i \sim \mathscr N(0,I; \mathscr F_i)}{\left| 
	\mathbf f_i D[\mathbf V_i](\mathbf x) \mathbf u \right|}
.
\]
Since $\mathbf f_i D[\mathbf V_i](\mathbf x) \mathbf u$ is a Gaussian with
zero average and variance $\|\mathbf u\|_{i,\mathbf x}^2$, 
\[
h_{C_i}(\mathbf u) = 
\frac{1}{\sqrt{2\pi}} \|\mathbf u\|_{i,\mathbf x}
.
\]
\end{proof}

\subsection{Proof of Theorem~\ref{integral1}}
The following technical fact will be needed:
\begin{lemma}\label{tech-lemma}
	\begin{equation}\label{volumes}
	\frac{n!}{(2\pi)^n} \vol_n(B^n) \vol_n (\mathbb R \mathbb P^n) = 1
	\end{equation}
\end{lemma}
\begin{proof}
Recall that
\[
\vol_n(B^n) = \frac{\pi^{\frac{n}{2}}}{\Gamma\left(\frac{n}{2}+1 \right)}
	\hspace{1em}	\text{and} \hspace{1em}	
	\vol_n (\mathbb R \mathbb P^n) =
	\frac{1}{2} \vol_n(S^n) = \frac{n+1}{2}\frac{\pi^{\frac{n+1}{2}}}{\Gamma\left(\frac{n+1}{2}+1 \right)}
.
\]
	The left-hand-side of \eqref{volumes} is therefore
\[
	\frac{(n+1)!}{2^{n+1}} 
	\frac{\sqrt{\pi}}
	{\Gamma\left(t \right)
	\Gamma\left(t+\frac{1}{2} \right)}
.
\]
	with $t = \frac{n}{2}+1$. Legendre's duplication formula
	yields
\[
	\Gamma\left(t \right)
	\Gamma\left(t+\frac{1}{2} \right) = 
	\frac{\sqrt{\pi}}{2^{n+1}}\Gamma(2t)
=
	\frac{(n+1)!\sqrt{\pi}}{2^{n+1}}
,
\]
completing the proof.
\end{proof}

\begin{proof}[Proof of Theorem~\ref{integral1}]
Lemma~\ref{MV-integral} with $\mathscr F_1 = \dots = \mathscr F_n = \mathscr F$ becomes
\[
	E_{\mathscr D}(\mathscr F) = \frac{n!}{\sqrt{2\pi}^n}
	\int_{\mathscr D} \dd \mathbb R^n V_n(C(\mathbf x))
\]
	where $C(\mathbf x)$ has support function
\[
	h_{C(\mathbf x)}(\mathbf u) = 
\frac{1}{\sqrt{2\pi}} \|\mathbf u\|_{\mathbf x}
=
	\|D[V](\mathbf x) \mathbf u\|
\]
	From Equation~\eqref{volC}, the convex body $C(\mathbf x)$ is an ellipsoid with volume
	\[
		\frac{\sqrt{ \det(D[V](\mathbf x)^*\  D[V](\mathbf x) ) }}
	{\sqrt{2\pi}^n} 
		\vol_n B^n.
	\]
Removing the volume of the $n$-ball from the integral, we are left
with the Jacobian of the Veronese embedding. Thus,
	\begin{eqnarray*}
		E_{\mathscr D}(\mathscr F) &=& 
	\frac{n! \vol_n(B^n)}{(2\pi)^n}
	\int_{\mathscr D} \sqrt{ \det(D[V](\mathbf x)^*\, D[V](\mathbf x) ) } \dd \mathbb R^n(x)
\\
	&=&
		\frac{n! \vol_n(B^n)}{(2\pi)^n} \vol_n (\changed{\mathscr V}).
	\end{eqnarray*}
	From Lemma~\ref{tech-lemma},
\[
E_{\mathscr D}(\mathscr F) = 
		\frac{\vol_n (\mathscr V)}{\vol_n (\mathbb R\mathbb P^{n})}
.
\]
\end{proof}

\section{Additivity Lemma and proof of main results}

The main theorems in this note hold for function spaces with a monomial basis.
This differs from the construction in \ocite{Fewspaces} which deals
with certain spaces of 
complex analytic functions. The \changed{lemma} below holds in both cases,
a general complex analytic setting (with the Hermitian metric associated to
a Kähler structure) and a
real setting (monomial bases required). We prove the real setting below.

\begin{lemma}\label{additivity-lemma}
	Assume that 
$\mathscr E_{A,\alpha}= 
\mathscr E_{B, \beta} \mathscr E_{C,\gamma}$.
	Then, $
\langle \cdot , \cdot \rangle_{\mathscr A,\mathbf x}
=
\langle \cdot , \cdot \rangle_{\mathscr B,\mathbf x}+
\langle \cdot , \cdot \rangle_{\mathscr C,\mathbf x}
$.
\end{lemma}
\begin{proof}
	For any space $\mathscr A= \mathscr E_{A,\alpha}$ of exponential sums, consider the convex
potential
\[
	\varphi_{\mathscr A}(\mathbf x) = \frac{1}{2} \log(\| \mathbf V_{\changed{\mathscr A}}(\mathbf x)\|^2)
.
\]
	Its derivative 
\[
	\mathbf m_{\mathscr A}(\mathbf x) \defeq D\varphi_{\mathscr A}(\mathbf x) = 
	\sum_{\mathbf a \in A}
	\frac{\alpha_{\mathbf a}^2 
	e^{2\mathbf a \mathbf x}}{\|\changed{\mathbf V_{\mathscr A}}(\mathbf x)\|^2}
	\mathbf a
\]
is known as the {\em momentum map}. We compute the second derivative:
\[
		\frac{1}{2}D^2 \varphi_{\mathscr A}(\mathbf x): \mathbf u, \mathbf v \mapsto 
	\sum_{\mathbf a \in A}
	\frac{\alpha_{\mathbf a}^2 
	e^{2\mathbf a \mathbf x}}{\|\changed{\mathbf V_{\mathscr A}}(\mathbf x)\|^2}
	(\mathbf a \mathbf u) (\mathbf a \mathbf v) 
	-
	(\mathbf m_{\mathscr A}(\mathbf x)\mathbf u) 
	(\mathbf m_{\mathscr A}(\mathbf x)\mathbf v) 
\]
	\changed{The right-hand side is precisely equal to 
\[
	\frac{1}{\|\mathbf V_{\mathscr A}(\mathbf x)\|^2} 
	\mathbf v^*D\mathbf V_{\mathscr A}(\mathbf x)^* 
	\left(I-\frac{1}{\|\mathbf V_{\mathscr A}(\mathbf x)\|^2} \mathbf V_{\mathscr A}(\mathbf x) \mathbf V_{\mathscr A}(\mathbf x)^*\right) D\mathbf V_{\mathscr A}(\mathbf x) \mathbf u. 
\]
Hence, we deduce the identity
	\[
\langle \mathbf u, \mathbf v \rangle_{\mathcal A, \mathbf x}
=
	\frac{1}{2}D^2 \varphi_{\mathscr A}(\mathbf x) (\mathbf u, \mathbf v)
.
	\]}
	Now, let $\mathscr A=\mathscr E_{A,\alpha}$, $\mathscr B=\mathscr E_{B,\beta}$, $\mathscr C=\mathscr E_{C,\gamma}$. Since
	\[
		\| \mathbf V_{\mathscr A}\|^2
		=
		\|  \mathbf V_{\mathscr B}\|^2
		\|  \mathbf V_{\mathscr C}\|^2,
	\]
	we have $
	\varphi_{\mathscr A}(\mathbf x)=
	\varphi_{\mathscr B}(\mathbf x)+
	\varphi_{\mathscr C}(\mathbf x)$,
hence
	$\mathbf m_{\mathscr A}(\mathbf x)=
	\mathbf m_{\mathscr B}(\mathbf x)+
	\mathbf m_{\mathscr C}(\mathbf x)$, and finally
	\[
\langle \cdot , \cdot \rangle_{\mathscr A,\mathbf x}
=
\langle \cdot , \cdot \rangle_{\mathscr B,\mathbf x}+
\langle \cdot , \cdot \rangle_{\mathscr C,\mathbf x}
.
\]
	\end{proof}

\begin{proof}[Proof of Theorem~\ref{gen-square-root}]
For each $i$ and for each $\mathbf x \in \mathscr D \changed{\, \subset \,} \mathbb R^n$,
	let $C_i(\mathbf x)$ be the convex body with 
support function
	\[
		h_{C_i(\mathbf x)}(\mathbf u) = \frac{1}{\sqrt{2\pi}} \| \mathbf u\|_{i, \mathbf x}
	\]
	where $\| \cdot \|_{i, \mathbf x} = \sqrt{\langle \cdot, \cdot \rangle_{i, \mathbf x}}$ is the norm associated to $\mathscr F_i$.
	By Lemma~\ref{MV-integral},
\[
	E_{\mathscr D}(\mathscr F_1, \dots, \mathscr F_n) = 
	\frac{n!}{\sqrt{2\pi}^{n}}
	\int_{\mathscr D} 
	\changed{\dd}\mathbb R^n(\mathbf x) \,
\mathrm{MV}_n( C_1(\mathbf x), \dots, C_n(\mathbf x))
.
\]
	By Lemma~\ref{additivity-lemma}, the inner product associated to
	$\mathscr F_i^d$ is $d\, \langle \cdot, \cdot \rangle_{i, \mathbf x}$,
	and the associated norm is $\sqrt{d}\, \| \cdot \|_{i, \mathbf x}$.
	Lemma~\ref{MV-integral} applied to $\mathscr F_1^{d_1}, \dots,
	\mathscr F_n^{d_n}$ yields:
\[
	E_{\mathscr D}(\mathscr F_1^{d_1}, \dots, \mathscr F_n^{d_n}) = 
	\frac{n!}{\sqrt{2\pi}^{n}}
	\int_{\mathscr D} 
	\changed{\dd}\mathbb R^n(\mathbf x) 
	\, \mathrm{MV}_n\left( \sqrt{d_1} C_1(\mathbf x), \dots, \sqrt{d_n} C_n(\mathbf x)\right)
\]
	because $\sqrt{d_1} C_1(\mathbf x)$ has support function
	\[
		h_{\sqrt{d_i} C_i(\mathbf x)}(\mathbf u) = \frac{\sqrt{d_i}}{\sqrt{2\pi}} \| \mathbf u\|_{i, \mathbf x}.
	\]
By the linearity property of the mixed volume,
	\begin{eqnarray*}
		E_{\mathscr D}(\mathscr F_1^{d_1}, \dots, \mathscr F_n^{d_n}) &=& 
		\frac{n! \sqrt{d_1 d_2 \changed{\cdots} d_n}}{\sqrt{2\pi}^{n}}
	\int_{\mathscr D} 
		\changed{\dd}\mathbb R^n(\mathbf x) \,
	\mathrm{MV}_n( C_1(\mathbf x), \dots, C_n(\mathbf x))\\
		&=&
		\sqrt{d_1 d_2 \changed{\cdots} d_n} E_{\mathscr D}(\mathscr F_1, \dots, \mathscr F_n) 
.
	\end{eqnarray*}
\end{proof}

\begin{remark}
	In the complex case described in \ocite{Fewspaces}, we do not compute
	the second derivative of the potential. Instead,
the Fubini Study metric is
	\[
	\omega_{\mathbf x} = \frac{\sqrt{-1}}{2} \partial \bar \partial
	\log K(\mathbf x, \mathbf x).
	\]
	Therefore, if $K(\mathbf x, \mathbf x) = \sum_{\mathbf a} |\psi_{\mathbf a}(\mathbf x)|^2$ with the $\psi_{\mathbf a}(\mathbf x)$ {\em analytic},
	only the holomorphic first derivatives of the 
	$\psi_{\mathbf a}$ (and the anti-holomorphic first derivatives 
	of $\bar \psi_{\mathbf a}$) appear in the expression of 
	$\omega_{\mathbf x}$ and its associated Hermitian metric.
\end{remark}

\begin{proof}[Proof of Theorem~\ref{sub-additive}]
	Let $\langle \cdot , \cdot \rangle_{i, \mathbf x}$ be the inner produc
	associated to $\mathscr F_i$, with $i=1,\dots, n$ and $\mathscr F_n =
	\mathscr G \mathscr H$. By Lemma~\ref{additivity-lemma},
\[
\langle \cdot , \cdot \rangle_{n, \mathbf x} = 
\langle \cdot , \cdot \rangle_{\mathscr G, \mathbf x} 
	+ \langle \cdot , \cdot \rangle_{\mathscr H, \mathbf x}
\]
where the two right inner products are the ones associated to $\mathscr G$ and
$\mathscr H$. In particular, associated norms satisfy
	\[
\| \mathbf u\|_{n, \mathbf x}^2
	\changed{=}
\| \mathbf u\|_{\mathscr G, \mathbf x}^2
+
\| \mathbf u\|_{\mathscr H, \mathbf x}^2
.
\]
	By the \changed{Triangle} inequality,
	\begin{equation}\label{Triangle}
\| \mathbf u\|_{n, \mathbf x}
\le
\| \mathbf u\|_{\mathscr G, \mathbf x}
+
\| \mathbf u\|_{\mathscr H, \mathbf x}
	\end{equation}

	Let $C_n(\mathbf x)$, $C_{\mathscr G}(\mathbf x)$ and $C_{\mathscr H}(\mathbf x)$ be the convex bodies with respective support functions
\begin{eqnarray*}
	h_{C_i(\mathbf x)}(\mathbf u) &=& \frac{1}{\sqrt{2\pi}} \| \mathbf u\|_{i, \mathbf x} \hspace{1em},\\
	h_{C_{\mathscr G}(\mathbf x)}(\mathbf u) &=& \frac{1}{\sqrt{2\pi}} \| \mathbf u\|_{\mathscr G, \mathbf x} \hspace{1em}\text{, and}\\
	h_{C_{\mathscr H}(\mathbf x)}(\mathbf u) &=& \frac{1}{\sqrt{2\pi}} \| \mathbf u\|_{\mathscr H, \mathbf x}\hspace{1em}.
\end{eqnarray*}
	Then 
\[
C_{n}(\mathbf x) \changed{\, \subset \,} C_{\mathscr G}(\mathbf x) + C_{\mathscr H}(\mathbf x)
\]

From the monotonicity and linearity properties of the mixed volume,
	\begin{eqnarray*}
	\mathrm{MV}_n(C_1(\mathbf x), \dots, C_n(\mathbf x))
		&\le&
	\mathrm{MV}_n(C_1(\mathbf x), \dots, C_{n-1}(\mathbf x),  C_{\mathscr G}(\mathbf x) + C_{\mathscr H}(\mathbf x))\\
		&\le&
	\mathrm{MV}_n(C_1(\mathbf x), \dots, C_{n-1}(\mathbf x),  C_{\mathscr G}(\mathbf x) )\\
	&&
	+
		\mathrm{MV}_n(C_1(\mathbf x), \dots, C_{n-1}(\mathbf x),  C_{\mathscr H}(\mathbf x)).
	\end{eqnarray*}
Theorem~\ref{sub-additive} follows now directly from Lemma~\ref{MV-integral}.
\end{proof}
\changed{
	\begin{remark} If the two norms on the right-hand side of \eqref{Triangle} are multiple of each other, then \eqref{Triangle} becomes an equality.
The same happens in the statement of Theorem~\ref{sub-additive}.
\end{remark}
}
\renewcommand{\MR}[1]{}

\end{document}